\documentclass{article}
\usepackage{miyamath,mathrsfs,amssymb,amsthm,amsmath,main,fourier}

\newtheoremstyle{normal}{3mm}{2mm}{\normalfont}{4mm}{\scshape}{.--- }{ }{}
\theoremstyle{normal}
\newtheorem{keyobservation}{Observation}

\begin{document}

\title{An Exponential Sum and Higher-Codimensional Subvarieties\\ of Projective Spaces over Finite Fields.}
\author{Kazuaki Miyatani and Makoto Sano}
%\address{Department of Mathematics, Graduate School of Science, Hiroshima University\\
%1-3-1 Kagamiyama, Higashi-Hiroshima 739-8526, Japan}
%\email{miyatani@hiroshima-u.ac.jp}
%\author{Makoto Sano}
%\address{Department of Mathematics, Graduate School of Science, Hiroshima University\\
%1-3-1 Kagamiyama, Higashi-Hiroshima 739-8526, Japan}
%\email{m126552@hiroshima-u.ac.jp}
\date{\empty}

%\subjclass{Primary~11T23; Secondary~14G05}
%\keywords{exponential sum; counting rational points}

\makeatletter
\def\thefootnote{($\ast$\@alph\c@footnote)}
\makeatother

\maketitle

\begin{abstract}
    A general method to express in terms of Gauss sums the number of rational points of 
    general subschemes of projective schemes over finite fields
    is applied to the image of the triple embedding $\bP^1\hookrightarrow\bP^3$.
    As a consequence, we obtain a non-trivial description of the value of a Kloosterman-sum-like exponential sum.
\end{abstract}

\section*{Introduction.}

Based on a calculation concerning the diagonal hypersurfaces and Jacobi sums,
Andr\'e Weil \cite{Weil:NSEFF} observed that the number of rational points 
of algebraic varieties over finite fields is a highly geometric information.
His celebrated Weil conjecture is eventually proved by Pierre Deligne \cite{Deligne:WeilI}
with Grothendieck's theory of \'etale cohomology.

Another, and the first, proof by Bernard Dwork \cite{Dwork:RZFAV} of the rationality of zeta function,
a part of Weil conjecture, more directly concerns exponential sums.
In the very beginning of his proof, he reduces the rationality of the zeta function of 
algebraic varieties to that of hypersurfaces of $(\bG_{\rmm})^n$,
and writes the number of rational points of the hypersurface by using exponential sums.

Besides general theories, expressing the number of rational points using exponential sums
has proved to be effective also in studying concrete projective hypersurfaces;
the study of Dwork families, for example by Neal Koblitz \cite{Koblitz},
is one of the most outstanding success of this strategy.
It seems therefore natural to expect the efficacy of this tactic in studying projective varieties of higher codimension.
This point of view, however, does not seem to be taken note of enough
in the study of concrete algebraic varieties of higher codimension.

In the former half of this article, we explain the general method for obtaining a formula
for the number of rational points of projective varieties in terms of Gauss sums.
Then, in the latter half, we apply the method to get an expression for 
the number of rational points of the image $V$ of the triple embedding $\bP^1\hookrightarrow \bP^3$.
Equating the result with the obvious equation $\# V(\bF_q)=q+1$, we get a non-trivial equation on an exponential sum
(Corollary \ref{cor:valueofcharsum});
\[
    \sum_{t_1,t_2,t_3\in\bF_q^{\times}}\theta\big(t_1+t_2+t_3-t_1^2t_2^{-2}t_3-t_1^2t_2^{-3}t_3^2-t_1t_2^{-2}t_3^2\big)
    = 2q^2-3q-1,
\]
where $\theta$ denotes a non-trivial character on $\bF_q$.
This result is also viewed as a description of the Frobenius trace on a pull-back of an 
Artin--Schreier sheaf by an appropriate morphism $(\bG_{\rmm})^3\to\bA^1$.
In the last section, we work on another subscheme of projective spaces,
the image of the Segre embedding $\bP^1\times\bP^2\hookrightarrow\bP^5$,
to see that the character sum $L$ above again appears in this calculation.

The authors expect that the content of this article can help the future study of the following two objects.
The first is the zeta function of concrete varieties;
the method described in this article is expected to make it possible to compute zeta function of
possibly non-rational varieties of higher codimension in projective spaces.
The second is the Kloosterman-sum-like exponential sums;
an explicit calculation of them may be possible by reducing the problem to finding a variety
the number of whose rational points is expressed by using the exponential sum
but is already known in some other methods.

\section*{Conventions and Notations.}

Throughout this article, we fix a prime number $p$, a power $q$ of $p$, and a prime number $l$ different from $p$.

The $n$-dimensional projective space $\bP^n_{\bF_q}$ 
\resp{affine space $\bA^n_{\bF_q}$} over $\bF_q$ is simply denoted by $\bP^n$ \resp{$\bA^n$},
and similarly, the torus $\bG_{\rmm,\bF_q}=\Spec\bF_q[t,t^{-1}]$ is denoted by $\bG_{\rmm}$.
Unless otherwise stated, the term ``rational point'' refers to ``$\bF_q$-rational point''.

The group of multiplicative characters $\bF_q^{\times}\to\overline{\bQ}_l^{\times}$ is denoted by
$\widehat{\bF_q^{\times}}$.
The trivial character is denoted by $\epsilon$.
Each character $\chi\in\widehat{\bF_q^{\times}}$, possibly trivial one, is also regarded as a function on $\bF_q$ by putting
$\chi(0)=0$.

Finally, {\em we fix a non-trivial additive character $\theta\colon\bF_q\to\overline{\bQ}_l^{\times}$ throughout this article.}

\section{Preliminaries on Gauss sum.}

In this section, we briefly recall the definition and some basic properties of Gauss sums
used in this article.
Note that we have fixed a non-trivial additive character $\theta\colon \bF_q\to\overline{\bQ}_l^{\times}$.

\begin{definition}
    For a character $\chi\in\widehat{\bF_q^{\times}}$, the Gauss sum of $\chi$ is defined to be
    \[
        G_\theta(\chi):=\sum_{x\in\bF_q^{\times}}\theta(x)\chi(x).
    \]
    Although the Gauss sum depends on the choice of $\theta$, we usually denote it simply by $G(\chi)$.
\end{definition}

The following three equations on characters are fundamental in doing calculations involving Gauss sums.

\begin{lemma}
    \label{prop:orthogonal}
    The following equations hold.

    {\rm (i)} For each $x\in\bF_q$,
\[
    \sum_{w\in\bF_q}\theta(wx) = \begin{cases} 0 & \text{if } x\neq 0, \\ q & \text{if } x=0.\end{cases}
\]

    {\rm (ii)} For each $\chi\in\widehat{\bF_q^{\times}}$,
\[
    \sum_{x\in\bF_q^{\times}}\chi(x) = \begin{cases} 0 & \text{if } \chi\neq\epsilon, \\ q-1 & \text{if } \chi=\epsilon.\end{cases}
\]

    {\rm (iii)} For each $x\in\bF_q^{\times}$,
    \[
        \sum_{\chi\in \widehat{\bF_q^{\times}}} \chi(x) = \begin{cases} 0 & \text{if } x\neq 1,\\
            q-1 & \text{if } x=1.\end{cases}
    \]
\end{lemma}

Now, we list two basic facts on Gauss sums.

\begin{proposition}
    \label{prop:gaussprodwithbar}
    For each character $\chi\in\widehat{\bF_q^{\times}}$,
    \[
        G(\chi)G\big(\chi^{-1}\big)=\begin{cases} q\chi(-1) & \text{ if } \chi\neq\epsilon,\\ 1 & \text{ if } \chi=\epsilon.
        \end{cases}
    \]
\end{proposition}

\begin{proposition}
    \label{prop:thetabygauss}
    For each $x\in\bF_q^{\times}$,
    \[
        \theta(x) = \frac{1}{q}\sum_{\chi\in\widehat{\bF_q^{\times}}}G\big(\chi^{-1}\big)\chi(x).
    \]
\end{proposition}

The detailed proofs for these propositions are omitted;
they are straightforward by using Lemma \ref{prop:orthogonal} (i) for Proposition \ref{prop:gaussprodwithbar},
and by using Lemma \ref{prop:orthogonal} (ii) for Proposition \ref{prop:thetabygauss}.

\section{Computing the Number of Rational Points --- General Strategy.}

In this section, we describe, in a general setting, a strategy to express
the number of rational points of subschemes of projective spaces over $\bF_q$ in terms of Gauss sums.
The first crucial point is the following reduction process to hypersurfaces;
this process is taken in the proof \cite[Theorem 1]{Dwork:RZFAV} of rationality of zeta functions,
and is now also used in algorithmic theory of zeta functions \cite[3]{Wan}.

\begin{keyobservation}
    \label{key}
    \emph{We obtain a formula for the number of rational points of a closed subscheme of a projective space
    if we obtain a formula for the number of rational points of some appropriate hypersurfaces.}

    In fact, let $V$ be a projective variety of which we want a formula for the number of rational points.
    Take a closed immersion of $V$ in a projective space $\bP^n$,
    and let the image be defined by $m$ homogeneous polynomials $f_1,\dots,f_m$.
    For each choice of integers $1\leq i_1<\dots<i_r\leq m$, we set
    \[
        N_{i_1,\dots,i_r} := \#\Set{ x \in \bP^n(\bF_q) |
            \exists j\in\{i_1,\dots,i_r\},\,\, f_j(x) = 0}.
    \]
    Then, since
    \[
        \#V(\bF_q) = \sum_{r=1}^m \sum_{1\leq i_1<\dots<i_r\leq m}(-1)^{r+1} N_{i_1,\dots,i_r},
    \]
    it suffices to obtain a formula for each $N_{i_1,\dots,i_r}$,
    which is exactly the number of rational points of the hypersurface defined by the
    product $f_{i_1}f_{i_2}\dots f_{i_r}$.
\end{keyobservation}

\begin{remark}
    \rm
    A formula for rational points of not necessarily closed subschemes of projective spaces is also obtained
    because such a subscheme is the difference of two closed subschemes.
\end{remark}

For obtaining a formula for the number of rational points of a hypersurface of $\bP^n$,
it suffices to count the number of rational points of the hypersurface of $\bA^{n+1}$ defined by
the same polynomial.
Moreover, the problem reduces to obtaining a formula for
\[
    \#\Set{ (x_1,\dots,x_{n+1}) \in (\bG_{\rmm})^{n+1} | [x_1:\dots:x_{n+1}]\in V(\bF_q)}
\]
because the locus with at least one coordinates being zero is covered by
hypersurfaces of lower-dimensional affine plane.

\begin{remark}
    \rm
When dealing with a specific projective variety $V\subset\bP^n$ (or equivalently, specific homogeneous
polynomials $f_1,\dots,f_m$), it is often easier to compute firstly the number
\[
    N_0 := \#\Set{ [x_1:\dots:x_{n+1}]\in V(\bF_q) | \exists i\in\{1,\dots,n+1\},\,\, x_i=0}
\]
directly, and to compute secondly, for each choice $1\leq i_1<\dots<i_r\leq m$, the number
\begin{align*}
    N_{i_1,\dots,i_r}^{\times} := \#\big\{\, & [x_1:\dots:x_{n+1}] \in \bP^n(\bF_q) \,\big|\\
    & \hspace{1em} \forall i\in\{1,\dots,n+1\},\,\,x_i\neq 0 \,\,\text{and}\,\,\exists j\in\{i_1,\dots,i_r\},\,\,f_j(x)=0\,\big\},
\end{align*}
which is $1/(q-1)$ times the number of rational points of the hypersurface of $(\bG_{\rmm})^{n+1}$ defined by 
the product $f_{i_1}f_{i_2}\dots f_{i_r}$. Under this notation, our formula becomes
\[
    \# V(\bF_q) = N_0+\sum_{r=1}^m\sum_{1\leq i_1<\dots<i_r\leq m}(-1)^{r+1}N_{i_1,\dots,i_r}^{\times}.
\]
Our calculations in Sections 3 and 4 also go in this manner.
\label{rem:specific}
\end{remark}

We conclude this section by giving a formula for the number of $\bF_q$-rational points of
an arbitrary hypersurface of $(\bG_{\rmm})^n$ in terms of Gauss sums.

Let us introduce a convention which we need to state the result.
Let $M=(m_{ij})_{i,j}$ be an $n\times N$ matrix with coefficients in $\bZ$.
Then, $M$ naturally defines the group homomorphism
$\varphi(M)\colon \big(\widehat{\bF_q^{\times}}\big)^N\to \big(\widehat{\bF_q^{\times}}\big)^n$,
explicitly expressed as
\[
    \varphi(M) \big((\chi_i)_{i=1,\dots,N}\big) = \left( \chi_1^{m_{j1}}\dots\chi_N^{m_{jN}}\right)_{j=1,\dots,n}.
\]
We always regard elements of $\big(\widehat{\bF_q^{\times}}\big)^N$ and
$\big(\widehat{\bF_q^{\times}}\big)^n$ as column vectors.

The following proposition gives a formula we are asking for.

\begin{proposition}[\cite{Delsarte}, \cite{Gomida}]
    \label{prop:hypersurfaceofgm}
    Let $n$ and $N$ be positive integers, let $c_1,\dots,c_N$ be elements of $\bF_q^{\times}$, and
    let $R=(r_{ij})_{i,j}\in M_{n,N}(\bZ)$ be an $n\times N$ matrix.
    Define a polynomial $f(X_1,\dots,X_n)\in\bF_q[X_1,\dots,X_n]$ by
    \[
        f(X_1,\dots,X_n)=\sum_{j=1}^Nc_jX_1^{r_{1j}}\dots X_n^{r_{nj}}.
    \]
    Then, the number of $n$-tuples $(x_1,\dots,x_n)\in(\bF_q^{\times})^n$ satisfying $f(x_1,\dots,x_n)=0$ equals
    \[
        \frac{(q-1)^n}{q} + \frac{(q-1)^{n+1-N}}{q}\sum_{ {}^{\rt}(\chi_1,\dots,\chi_N)\in \Ker(\varphi(\widetilde{R}))}\prod_{j=1}^NG\big(\chi_j^{-1}\big)\chi_j(c_j),
    \]
    where
    \[
        \widetilde{R} = \left(\begin{array}{ccc} \multicolumn{3}{c}{\raisebox{-4pt}{\large $R$}}\\ && \\ \hline 1 & \dots & 1\end{array}\right).
    \]
\end{proposition}

\begin{proof}
    Lemma \ref{prop:orthogonal} (i) shows that
    \begin{align*}
        & \sum_{x_1,\dots,x_n\in\bF_q^{\times}}\sum_{w\in\bF_q}\theta\big(wf(x_1,\dots,x_n)\big)\\
        = & q\#\Set{(x_1,\dots,x_n)\in(\bF_q^{\times})^n | f(x_1,\dots,x_n)=0},
    \end{align*}
    thus it suffices to calculate the left-hand side and divide the result by $q$.
    Because $\theta(0)=1$ and because $\theta$ transforms addition to multiplication,
    the left-hand side equals
    \begin{align*}
         & (q-1)^n+\sum_{x_1,\dots,x_n,w\in\bF_q^{\times}}\theta\big(wf(x_1,\dots,x_n)\big)\\
         = & (q-1)^n+\sum_{x_1,\dots,x_n,w\in\bF_q^{\times}}\prod_{j=1}^N\theta\big(c_jwx_1^{r_{1j}}\dots x_n^{r_{nj}}\big).
     \end{align*}
    Now, Proposition \ref{prop:thetabygauss} shows that the second term equals
    \[
        \frac{1}{(q-1)^N}\sum_{x_1,\dots,x_n,w\in\bF_q^{\times}}\prod_{j=1}^N\sum_{\chi_j\in\widehat{\bF_q^{\times}}}G\big(\chi_j^{-1}\big)\chi_j\big(c_jwx_1^{r_{1j}}\dots x_n^{r_{nj}}\big)
    \]
    and by using the multiplicativity of characters, we know that it equals
    \begin{align*}
        &\frac{1}{(q-1)^N}\sum_{\chi_1,\dots,\chi_N\in\widehat{\bF_q^{\times}}}\Bigg\{\prod_{j=1}^{N}G\big(\chi_j^{-1}\big)\chi_j(c_j)\\ & \hspace{12em} \prod_{i=1}^n \sum_{x_i\in\bF_q^{\times}}\big(\chi_1^{r_{i1}}\dots\chi_N^{r_{iN}}\big)(x_i) \sum_{w\in\bF_q^{\times}}(\chi_1\dots \chi_N)(w)\Bigg\}.
    \end{align*}
    Because of Lemma \ref{prop:orthogonal} (ii), each summand with respect to the outer sum vanishes
    unless $\chi_1^{r_{i1}}\dots\chi_N^{r_{iN}}=\epsilon$ for each $i=1,\dots,n$ and $\chi_1\dots\chi_N=\epsilon$,
    that is, unless ${}^{\rt}(\chi_1,\dots,\chi_N)\in\Ker\varphi(\widetilde{R})$;
    if it is the case, the summand equals $(q-1)^{n+1}\prod_{j=1}^NG\big(\chi_j^{-1}\big)\chi_j(c_j)$.
    This completes the proof.
\end{proof}

\section{The Image of Triple Embedding $\bP^1\hookrightarrow\bP^3$.} 
In this section, we compute the number of rational points of $\bP^1$
by identifying it as the image $V$ of triple embedding $\bP^1\hookrightarrow\bP^3$
and by following the strategy in the previous section.
Later, we compare the result with the obvious answer $\#\bP^1(\bF_q)=q+1$.

The variety $V$ is defined in $\bP^3=\Proj\big(\bF_q[x_1,x_2,x_3,x_4]\big)$ 
by the following three polynomials:
\[
	f_1(x)= x_1x_3-x_2^2,\quad f_2(x)= x_2x_4-x_3^2,\quad f_3(x)= x_1x_4-x_2x_3.
\]
We follow the notation in Remark \ref{rem:specific}; in particular, we have
\begin{equation}
    \# V(\bF_q) = N_0 + N_1^{\times} + N_2^{\times} + N_3^{\times} - N_{12}^{\times} - N_{13}^{\times} - N_{23}^{\times} + N_{123}^{\times}.
    \label{eq:summarize1}
\end{equation}

Firstly, we directly compute the number $N_0$.
In our case, we easily see that
\begin{align*}
    & \Set{[x_1:x_2:x_3:x_4]\in V(\bF_q) | \exists i\in\{1,2,3,4\},\,\, x_i=0}\\
    = & \big\{[1:0:0:0], [0:0:0:1]\big\};
\end{align*}
thus $N_0=2$.

Secondly, we compute the number
\[
    N_i^{\times} = \#\Set{ [x_1:x_2:x_3:x_4]\in \bP^4(\bF_q) | f_i(x)=0 \,\,\, \text{and} \,\,\, \forall j,\, x_j\neq 0} 
\]
for $i\in\{1,2,3\}$.
Although we could employ Proposition \ref{prop:hypersurfaceofgm}, a direct calculation is easy in this case.
Namely, if $i=1$, we may assume $x_4=1$, and after a free choice of $x_2,x_3\in\bF_q^{\times}$,
the remaining $x_1\in\bF_q^{\times}$ is uniquely determined;
thus we get $N_1^{\times}=(q-1)^2$. By a similar argument, we see that
\[
    N_1^{\times}=N_2^{\times}=N_3^{\times}=(q-1)^2.
\]

Thirdly, we compute the number
\[
    N_{ij}^{\times} = \#\Set{ [x_1:x_2:x_3:x_4] \in \bP^4(\bF_q) | f_i(x)f_j(x)=0 \,\,\, \text{and} \,\,\, \forall j,\, x_j\neq 0}
\]
for $(i,j)\in\{(1,2), (1,3), (2,3)\}$.
Since
\[
    f_1(x)f_2(x) = x_1x_2x_3x_4 - x_1x_3^3 - x_2^3x_4 + x_2^2x_3^2,
\]
Proposition \ref{prop:hypersurfaceofgm} shows that
%\begin{equation}
%    (q-1) N_{12}^{\times} = \frac{(q-1)^4}{q}+\frac{q-1}{q}\sum_{{}^{\rt}(\chi_1,\dots,\chi_4)\in\Ker\varphi(\widetilde{R_{12}})}G\big(\chi_1^{-1}\big)G\big(\chi_2^{-1}\big)G\big(\chi_3^{-1}\big)G\big(\chi_4^{-1}\big)\chi_2\chi_3(-1),
%    \label{eq:quadratic1}
%\end{equation}
\begin{equation}
    (q-1) N_{12}^{\times} = \frac{(q-1)^4}{q}+\frac{q-1}{q}\sum_{{}^{\rt}(\chi_1,\dots,\chi_4)\in\Ker\varphi(\widetilde{R_{12}})}\prod_{j=1}^4 G\big(\chi_j^{-1}\big)\cdot\chi_2\chi_3(-1),
    \label{eq:quadratic1}
\end{equation}
where
\[
    \widetilde{R_{12}} = \begin{pmatrix} 1 & 1 & 0 & 0 \\ 1 & 0 & 3 & 2 \\ 1 & 3 & 0 & 2 \\ 1 & 0 & 1 & 0 \\ 1 & 1 & 1 & 1\end{pmatrix}.
\]
Since an elementary linear algebra with coefficient in $\bZ/(q-1)\bZ$ 
shows that $\Ker\varphi(\widetilde{R_{12}})=\Set{ {}^{\rt}(\chi, \chi^{-1}, \chi^{-1}, \chi) | \chi\in\widehat{\bF_q^{\times}}}$,
we have
\[
    (q-1) N_{12}^{\times} = \frac{(q-1)^4}{q}+\frac{q-1}{q}\sum_{\chi\in\widehat{\bF_q^{\times}}}G\big(\chi^{-1}\big)G\big(\chi\big)G\big(\chi^{-1}\big)G\big(\chi\big).
\]
With the aid of Proposition \ref{prop:gaussprodwithbar}, we get (after dividing both sides by $q-1$)
\begin{align}
    \label{eq:121}
    N_{12}^{\times} & = \frac{(q-1)^3}{q}+\frac{1}{q}\left(\sum_{\chi\in\widehat{\bF_q^{\times}}\setminus\{\epsilon\}}\big(q\chi(-1)\big)^2+(-1)^4\right)\\
    & = \frac{(q-1)^3}{q}+\frac{1}{q}\left\{q^2(q-2)+1\right\} = 2q^2-5q+3. \nonumber
\end{align}
Similarly, since
\[
    f_2(x)f_3(x) = x_1x_2x_4^2-x_2^2x_3x_4-x_1x_3^2x_4+x_2x_3^3,
\]
$(q-1)N_{23}^{\times}$ equals (\ref{eq:quadratic1}) with $\widetilde{R_{12}}$ replaced by
\[
\widetilde{R_{13}} = \begin{pmatrix} 1 & 0 & 1 & 0 \\ 1 & 2 & 0 & 1 \\ 0 & 1 & 2 & 3 \\ 2 & 1 & 1 & 0 \\ 1 & 1 & 1 & 1\end{pmatrix}.
\]
In fact, we may directly show that $\Ker\varphi(\widetilde{R_{12}})=\Ker\varphi(\widetilde{R_{13}})$,
and therefore we have $N_{13}^{\times}=N_{23}^{\times}$. Moreover, the symmetry shows $N_{23}^{\times}=N_{13}^{\times}$, and we have
\[
    N_{12}^{\times}=N_{13}^{\times}=N_{23}^{\times}=2q^2-5q+3.
\]

Finally, we compute the number $N_{123}^{\times}$. Since
\[
    f_1(x)f_2(x)f_3(x) = x_1^2x_2x_3x_4^2 - x_1^2x_3^3x_4 + x_1x_2x_3^4 - x_1x_2^3x_4^2 + x_2^4x_3x_4 - x_2^3x_3^3,
\]
Proposition \ref{prop:hypersurfaceofgm} shows that $(q-1)N_{123}^{\times}$ equals
\begin{equation}
    \frac{(q-1)^4}{q} + \frac{1}{q(q-1)}\sum_{{}^{\rt}(\chi_1,\dots,\chi_6)\in\Ker\big(\varphi(\widetilde{R_{123}})\big)}
    \prod_{j=1}^6 G\big(\chi_j^{-1}\big)(\chi_2\chi_4\chi_6)(-1)
    \label{eq:1231}
\end{equation}
with
\[
    \widetilde{R_{123}} = \begin{pmatrix}2&2&1&1&0&0\\1&0&1&3&4&3\\1&3&4&0&1&3\\2&1&0&2&1&0\\1&1&1&1&1&1\end{pmatrix}.
\]
Again, an elementary linear algebra shows that $\Ker\varphi(\widetilde{R_{123}})$ equals
\begin{equation}
    \Set{ {}^{\rt}(\chi_4^{-2}\chi_5^{-2}\chi_6^{-1}, \chi_4^2\chi_5^3\chi_6^2, \chi_4^{-1}\chi_5^{-2}\chi_6^{-2}, \chi_4, \chi_5, \chi_6) | \chi_4,\chi_5,\chi_6\in\widehat{\bF_q^{\times}}}.
    \label{eq:kernel1}
\end{equation}
This time, we cannot simplify the sum in the second term of (\ref{eq:1231}) as we did in (\ref{eq:121}).
Instead, we rewrite this sum by using a more simple character sum.

By the definition of Gauss sums, the sum in the second term in (\ref{eq:1231}) equals
\[
    \sum_{ {}^{\rt}(\chi_1,\dots,\chi_6)\in\Ker\varphi(\widetilde{R_{123}})}
    \prod_{i=1}^6\left\{\sum_{t_i\in\bF_q^{\times}}\theta(t_i)\chi_i^{-1}(t_i)\right\}(\chi_4\chi_5\chi_6)(-1).
\]
Now, the description ($\ref{eq:kernel1}$) of $\Ker\varphi(\widetilde{R_{123}})$ shows that this equals
\begin{align*}
    & \sum_{\chi_4,\chi_5,\chi_6\in\widehat{\bF_q^{\times}}}\sum_{t_1,\dots,t_6\in\bF_q^{\times}}\left(\prod_{i=1}^6\theta(t_i)\right)\big(\chi_4^2\chi_5^2\chi_6\big)(t_1)\big(\chi_4^{-2}\chi_5^{-3}\chi_6^{-2}\big)(t_2)\big(\chi_4\chi_5^2\chi_6^2\big)(t_3)\\
    & \hspace{12em} \chi_4^{-1}(t_4)\chi_5^{-1}(t_5)\chi_6^{-1}(t_6)(\chi_4\chi_5\chi_6)(-1) \\
    = & \sum_{t_1,\dots,t_6\in\bF_q^{\times}}\theta\left(\sum_{i=1}^6 t_i\right)\sum_{\chi_4\in\widehat{\bF_q^{\times}}}\chi_4\big(-t_1^2t_2^{-2}t_3t_4^{-1}\big)\\
    & \hspace{12em}\sum_{\chi_5\in\widehat{\bF_q^{\times}}}\chi_5\big(-t_1^2t_2^{-3}t_3^2t_5^{-1}\big)
    \sum_{\chi_6\in\widehat{\bF_q^{\times}}}\chi_6\big(-t_1t_2^{-2}t_3^2t_6^{-1}\big).
\end{align*}
Applying Lemma \ref{prop:orthogonal} (iii), this equals
\[
    (q-1)^3\sum_{t_1,t_2,t_3\in\bF_q^{\times}}\theta\left(t_1+t_2+t_3-t_1^2t_2^{-2}t_3-t_1^2t_2^{-3}t_3^2-t_1t_2^{-2}t_3^2\right),
\]
and therefore, we get the equation
\[
    N_{123}^{\times} = \frac{(q-1)^3}{q}+\frac{q-1}{q}\sum_{t_1,t_2,t_3\in\bF_q^{\times}}\theta\left(t_1+t_2+t_3-t_1^2t_2^{-2}t_3-t_1^2t_2^{-3}t_3^2-t_1t_2^{-2}t_3^2\right).
\]

Now, we summarize the calculation and we get the result.

\begin{theorem}
    Let $L$ be the sum
    \[
        L = \sum_{t_1,t_2,t_3\in\bF_q^{\times}}\theta\left(t_1+t_2+t_3-t_1^2t_2^{-2}t_3-t_1^2t_2^{-3}t_3^2-t_1t_2^{-2}t_3^2\right).
    \]
    Then, $\# V(\bF_q)= -2q^2 + 6q - 1+ ((q-1)L-1)/q$.
    \label{thm:result1}
\end{theorem}
\begin{proof}
    Substituting the results above in the equation (\ref{eq:summarize1}), we get
    \[
        \# V(\bF_q) = 2 + 3(q-1)^2 - 3(2q^2-5q+3) + (q-1)^3/q + (q-1)L/q.
    \]
\end{proof}

As a corollary, we obtain an explicit description of the exponential sum $L$.

\begin{corollary} $L = 2q^2-3q-1$.
    \label{cor:valueofcharsum}
\end{corollary}
\begin{proof}
    Since $\# V(\bF_q) = \#\bP^1(\bF_q)=q+1$, Theorem \ref{thm:result1} gives
    an equation on $L$. Solving it gives the corollary.
\end{proof}

This corollary can be viewed as an explicit description of the trace function of a rank-one $l$-adic sheaf.

\begin{corollary}
    Define a morphism $f\colon \bG_{\rmm}^3\to \bA^1$ by
    \[
        f(t_1,t_2, t_3) = t_1+t_2+t_3-t_1^2t_2^{-2}t_3-t_1^2t_2^{-3}t_3^2-t_1t_2^{-2}t_3^2,
    \]
    and denote the Artin--Schreier sheaf on $\bA^1$ associated to $\theta$ by $\sL_{\theta}$.
    Then, the equation
    \[
        \sum_{i=0}^\infty (-1)^i \Tr\big(\Frob; H_{\rc}^i(\bG_{\rmm}^3\times_{\bF_q}\overline{\bF_q}, f^{\ast}\sL_{\theta})\big) = 2q^2-3q-1.
    \]
    holds.
\end{corollary}
\begin{proof}
    The left-hand side equals $L$ by the Grothendieck trace formula \cite[III B 1.3]{SGA5}.
\end{proof}

\begin{remark}
    \rm
    Corollary \ref{cor:valueofcharsum} can be also rewritten in terms of general hypergeometric functions 
    \cite[8.1]{GGR:HFAF}, \cite{GG:HFFF},
    which Lei Fu \cite{Fu:GKZ} calls GKZ hypergeometric sum.
    In fact, let $A$ be the $3\times 6$-matrix
    \[
        A = \begin{pmatrix} 1 & 0 & 0 & 2 & 2 & 1 \\ 0 & 1 & 0 & -2 & -3 & -2 \\ 0 & 0 & 1 & 1 & 2 & 2 \end{pmatrix},
    \]
    and let $\Hyp_{\theta}(x_1,x_2,x_3,x_4,x_5,x_6; \chi_1, \chi_2, \chi_3)$ denote
    the GKZ hypergeometric sum associated to $A$ with parameters $\chi_1,\chi_2,\chi_3$.
    Then the exponential sum $L$ coincides with the value $\Hyp_{\theta}(1,1,1,-1,-1,-1;\epsilon,\epsilon,\epsilon)$.
    In other words, let $\Hyp_{\theta}(\chi_1,\chi_2,\chi_3)$ denote the $l$-adic GKZ hypergeometric sheaf
    associated to the matrix $A$ with parameters $\chi_1,\chi_2,\chi_3\in\widehat{\bF_q^{\times}}$;
    this is a mixed perverse sheaf on $\bA^6$ of weights $\leq 9$ \cite[Theorem 0.3 (i)]{Fu:GKZ}.
    Then, setting $x=(1,1,1,-1,-1,-1)\in\bA^6$ and taking arbitrary geometric point $\overline{x}$ lying above $x$, we have
    \[
        \Tr\big(\Frob_x; \Hyp_{\theta}(\epsilon,\epsilon,\epsilon)_{\overline{x}}\big) = (-1)^9L = -2q^2+3q+1.
    \]
\end{remark}

\section{The Image of Segre Embedding $\bP^1\times\bP^2\hookrightarrow\bP^5$.}

In this section, we give another example of counting rational points,
the image $V$ of the Segre embedding $\bP^1\times\bP^2 \hookrightarrow\bP^5$,
and show that the character sum $L$ defined in Theorem $\ref{thm:result1}$ again appears in the calculation.
The variety $V$ is defined in $\bP^5=\Proj\big(\bF_q[x_1,x_2,\dots,x_6]\big)$
by the following three polynomials:
\[
	f_1(x)= x_1x_5-x_2x_4,\quad f_2(x)= x_1x_6-x_3x_4,\quad f_3(x)= x_2x_6-x_3x_5.
\]
We again work under the notations in Remark \ref{rem:specific}.
This time, the symmetry shows that $N_1^{\times}=N_2^{\times}=N_3^{\times}$ and that
$N_{12}^{\times}=N_{23}^{\times}=N_{13}^{\times}$; we therefore have
\begin{equation}
    \# V(\bF_q) = N_0 + 3N_1^{\times} - 3N_{12}^{\times} + N_{123}^{\times}.
    \label{eq:summarize2}
\end{equation}

Firstly, we compute $N_0$.
Although we may again describe all points in $V_0$
(or may, at the beginning, proceed not along Remark \ref{rem:specific} but along the original Observation \ref{key}),
it here seems easier to identify $V$ as $\bP^1\times\bP^2$.
For a subscheme $W$ of $\bP^n$, let $W_0$ denote the set of rational points of $W$
at least one of whose components is zero;
therefore $N_0=\# V_0$ under this notation.
The identification $V=\bP^1\times\bP^2$ gives
\[
    V_0=\left\{(\bP^1)_0\times\bP^2(\bF_q)\right\} \cup \left\{\bP^1(\bF_q)\times(\bP^2)_0\right\},
\]
and it is easy to see that $\#(\bP^1)_0 = 2$ and $\#(\bP^2)_0 = 3q$.
Now, we see that $N_0=2(q^2+q+1)+(q+1)3q-2\cdot 3q=5q^2-q+2$.

Secondly, the number $(q-1)N_1^{\times} = \#\big\{ x\in(\bF_q^{\times})^6 \,\big|\, f_1(x)=0\big\}$ 
can be calculated just as in the previous section, and the result is $N_1^{\times}=(q-1)^4$.

Thirdly, we compute the number $(q-1)N_{12}^{\times}=\#\big\{ x\in(\bF_q^{\times})^6 \,\big|\, f_1f_2(x) =0\big\}$.
Since
\[
	f_1(x)f_2(x) = x_1^2x_5x_6 - x_1x_3x_4x_5 - x_1x_2x_4x_6 + x_2x_3x_4^2, 
\]
Proposition \ref{prop:hypersurfaceofgm} shows that
\[
    (q-1)N_{12}^{\times} = \frac{(q-1)^6}{q}+\frac{(q-1)^3}{q}\sum_{ {}^{\rt}(\chi_1,\dots,\chi_4)\in\Ker\varphi(\widetilde{R_{12}})}\prod_{j=1}^4 G\big(\chi_j^{-1}\big)\cdot \chi_2\chi_3(-1),
\]
where
\[
    \widetilde{R_{12}} = \begin{pmatrix} 2 & 1 & 1 & 0 \\ 0 & 0 & 1 & 1 \\ 0 & 1 & 0 & 1 \\ 0 & 1 & 1 & 2 \\
        1 & 1 & 0 & 0 \\ 1 & 0 & 1 & 0 \\ 1 & 1 & 1 & 1\end{pmatrix}.
\]
Since $\Ker\varphi(\widetilde{R_{12}})=\Set{ {}^{\rt}(\chi, \chi^{-1}, \chi^{-1}, \chi) | \chi\in\widehat{\bF_q^{\times}}}$, we may calculate $N_{12}^{\times}$ as in (\ref{eq:121}) and we get
\[
    N_{12}^{\times} = \frac{(q-1)^5}{q}+\frac{(q-1)^2}{q}\big\{q^2(q-2)+1\big\}=(q-1)^2(2q^2-5q+3).
\]

Finally, we compute the number $(q-1)N_{123}^{\times}=\#\big\{x\in(\bF_q^{\times})^6 \, \big|\, f_1(x)f_2(x)f_3(x)$ $=0 \big\}$.
Since $f_1(x)f_2(x)f_3(x)$ equals
\[
    x_1^2x_2x_5x_6^2 - x_1^2x_3x_5^2x_6 + x_1x_3^2x_4x_5^2 - x_1x_2^2x_4x_6^2 + x_2^2x_3x_4^2x_6 - x_2x_3^2x_4^2x_5,
\]
Proposition \ref{prop:hypersurfaceofgm} shows that
\[
    (q-1) N_{123}^{\times} = \frac{(q-1)^6}{q} + \frac{q-1}{q}\sum_{{}^{\rt}(\chi_1,\dots,\chi_6)\in\Ker\big(\varphi(\widetilde{R_{123}})\big)}
    \prod_{j=1}^6 G\big(\chi_j^{-1}\big)(\chi_2\chi_4\chi_6)(-1)
\]
with
\[
    \widetilde{R_{123}} = \begin{pmatrix}2&2&1&1&0&0\\1&0&0&2&2&1\\0&1&2&0&1&2\\0&0&1&1&2&2\\1&2&2&0&0&1\\2&1&0&2&1&0\\1&1&1&1&1&1\end{pmatrix}.
\]
We may again show that $\Ker\varphi(\widetilde{R_{123}})$ equals
\[
    \Set{ {}^{\rt}(\chi_4^{-2}\chi_5^{-2}\chi_6^{-1}, \chi_4^2\chi_5^3\chi_6^2, \chi_4^{-1}\chi_5^{-2}\chi_6^{-2}, \chi_4, \chi_5, \chi_6) | \chi_4,\chi_5,\chi_6\in\widehat{\bF_q^{\times}}};
\]
this space is the same as what we treated in the previous subsection.
Therefore, we may proceed the calculation just as before and show that, under the definition of $L$ in Theorem \ref{thm:result1},
\[
    N_{123}^{\times} = \frac{(q-1)^5}{q}+\frac{(q-1)^3}{q}L.
\]

Let us summarize the calculations above;
substituting the results in (\ref{eq:summarize2}), we see that
\[
    \# V(\bF_q) = -2q^4+10q^3-12q^2+10q+1+\{(q-1)^3L-1\}/q.
\]
This equation and Corollary \ref{cor:valueofcharsum} give
\[
    \# V(\bF_q) = q^3+2q^2+2q+1 = (q+1)(q^2+q+1),
\]
which coincides with $\# (\bP^1\times\bP^2)(\bF_q)$.

\section*{Acknowledgements.}

This article is based on the master thesis of the second author.
He would like to express his appreciation to his advisor Shun-ichi Kimura for his close and consistent
guidance,
and to Takashi Ono who wrote the textbook \cite{Ono} through which he learned algebraic number theory.

The first author also would like to express his gratitude to Shun-ichi Kimura for his fruitful comments 
on this article and for his encouraging authors to publish this result.

\bibliographystyle{dagaz}
\bibliography{math}

\begin{thebibliography}{99}
\itemsep=0pt
\parskip=0pt
\small

\bibitem[1]{Deligne:WeilI}
P.~Deligne,
\newblock ``{\em La conjecture de {W}eil. {I}},''
\newblock Inst.\ Hautes \'Etudes Sci.\ Publ.\ Math. {\bf 43} (1974), 273--307.

\bibitem[2]{Delsarte}
J.~Delsarte,
\newblock ``{\em Nombre de solutions des \'equations polynomiales sur un corps
  fini},''
\newblock in S\'eminaire {B}ourbaki, {V}ol.\ 1 Exp.\ No.\ 39, 321--329. Soc.
  Math. France, Paris, 1995.

\bibitem[3]{Dwork:RZFAV}
B.~Dwork,
\newblock ``{\em On the rationality of the zeta function of an algebraic
  variety},''
\newblock Amer.\ J.\ Math. {\bf 82} (1960), 631--648.

\bibitem[4]{Gomida}
E. Furtado Gomida,
\newblock ``{\em On the theorem of {A}rtin--{W}eil},''
\newblock Bol.\ Soc.\ Mat.\ S{\~a}o Paulo {\bf 4} (1949, 1951), 1--18.

\bibitem[5]{Fu:GKZ}
L.~Fu,
\newblock ``{\em $\ell$-adic {GKZ} hypergeometric sheaf and exponential
  sums},''
\newblock arXiv:1208.1373v2.

\bibitem[6]{GG:HFFF}
I.~M. Gelfand and M.~I. Graev,
\newblock ``{\em Hypergeometric functions over finite fields},''
\newblock Dokl. Akad. Nauk {\bf 381}(6) (2001).
\newblock English Translation in Doklady Math.\ {\bf 64} (2001), 402--406.

\bibitem[7]{GGR:HFAF}
I.~M. Gelfand, M.~I. Graev and V.~S. Retakh,
\newblock ``{\em Hypergeometric functions over an arbitrary field},''
\newblock Uspekhi Mat. Nauk {\bf 59}(5(359)) (2004), 29--100.
\newblock English Translation in Russian Math.\ Surveys {\bf 59} (2004),
  831--905.

\bibitem[8]{SGA5}
A.~Grothendieck,
\newblock ``{\em Cohomologie $l$-adique et Fonctions {L}},''
\newblock Lecture Notes in Math. {\bf 589},
\newblock Springer-Verlag, 1977.


\bibitem[9]{Koblitz}
N.~Koblitz,
\newblock ``{\em The number of points on certain families of hypersurfaces over
  finite fields},''
\newblock Compositio Math. {\bf 48}(1) (1983), 3--23.

\bibitem[10]{Ono}
T.~Ono,
\newblock ``{\em S{\=u}ron Josetsu},''
\newblock Shokabo, 1987.
\newblock English Translation: ``An introduction to algebraic number theory,''
  The University Series in Mathematics, Plenum Press, New York, 1990.

\bibitem[11]{Wan}
D.~Wan,
\newblock ``{\em Algorithmic theory of zeta functions over finite fields},''
\newblock in Algorithmic number theory: lattices, number fields, curves and
  cryptography, Math.\ Sci.\ Res.\ Inst.\ Publ. {\bf 44},. Cambridge Univ.
  Press (2008), 551--578.

\bibitem[12]{Weil:NSEFF}
A.~Weil,
\newblock ``{\em Numbers of solutions of equations in finite fields},''
\newblock Bull.\ Amer.\ Math.\ Soc. {\bf 55} (1949), 497--508.

\end{thebibliography}

\end{document}